\documentclass[12pt, a4paper]{amsart}
\usepackage{hyperref, fullpage, amsmath, amsfonts, amsthm, amssymb, stmaryrd}

\theoremstyle{plain}
\newtheorem{Theorem}{Theorem}[section]
\newtheorem{Lemma}[Theorem]{Lemma}
\newtheorem{Corollary}[Theorem]{Corollary}
\newtheorem{Proposition}[Theorem]{Proposition}

\theoremstyle{remark}
\newtheorem{Remark}[Theorem]{Remark}

\newcommand{\hhb}[1]{{\hbox to#1pt{}}}
\newcommand{\notdivides}{\mbox{$\not|$}}

\address{Fachbereich Mathematik und Statistik, University of 
Konstanz, 78457 Konstanz, Germany}
\email{arno.fehm@uni-konstanz.de}

\begin{document}

\title{Existential $\emptyset$-definability of\\ henselian valuation rings}
\author{Arno Fehm}

\begin{abstract}
In \cite{AnscombeKoenigsmann}, Anscombe and Koenigsmann give an existential $\emptyset$-definition of
the ring of formal power series $F[[t]]$ in its quotient field in the case where $F$ is finite.
We extend their method in several directions
to give general definability results for henselian valued fields with finite or pseudo-algebraically closed residue fields.
\end{abstract}

\maketitle

\section{Introduction}

\noindent
The question of first-order definability of valuation rings in their quotient fields has a long history.
Given a valued field $K$, one is interested in whether there exists a first-order formula $\varphi$
in the language $\mathcal{L}=\{+,-,\cdot,0,1\}$ of rings such that the set $\varphi(K)$ defined by $\varphi$ in $K$
is precisely the valuation ring, and what complexity such formula must have.

Many results of this kind are known for {\em henselian} valued fields,
like fields of formal power series $K=F((t))$ over a field $F$,
and their valuation ring $F[[t]]$. 
In this setting, a definition going back to Julia Robinson gives an existential definition of the valuation ring using the parameter $t$.
Later, Ax \cite{Ax} gave a definition of the valuation ring, which uses no parameters, but is not existential.

Recently, Anscombe and Koenigsmann \cite{AnscombeKoenigsmann} succeeded to give an existential and pa\-ra\-meter-free definition of
$F[[t]]$ in $F((t))$ in the special case where $F=\mathbb{F}_q$ is a finite field.
Their proof uses the fact that $\mathbb{F}_q$ can be defined in $\mathbb{F}_q((t))$ by the quantifier-free formula $x^q-x=0$.
In particular, their result does not apply to any infinite field $F$, and their formula depends heavily on $q$.

In this note we simplify and extend their method.
As a first application we get the following general definability result for 
henselian valued fields with
finite or pseudo-algebraically closed residue fields (Theorem \ref{thm:finite} and Theorem \ref{thm:PAC}):

\begin{Theorem}
Let $K$ be a henselian valued field with valuation ring $\mathcal{O}$ and residue field $F$.
If $F$ is finite or pseudo-algebraically closed and the algebraic part of $F$ is not algebraically closed,
then there exists an $\exists$-$\emptyset$-definition of $\mathcal{O}$ in $K$.
\end{Theorem}

As a further application, in Section \ref{sec:uniform} we find
definitions of the valuation ring 
which are uniform for large (infinite) families of finite residue fields,
like the following one for finite prime fields (Theorem \ref{thm:uniform}):

\begin{Theorem}\label{thm:intro2}
For every $\epsilon>0$ there exists an $\exists$-$\emptyset$-formula $\varphi$ 
and a set $P$ of prime numbers of Dirichlet density at least $1-\epsilon$
such that for any henselian valued field $K$ with valuation ring $\mathcal{O}$ 
and residue field $F$ with $|F|\in P$,
the formula $\varphi$ defines $\mathcal{O}$ in $K$.
\end{Theorem}

In particular, this applies to power series fields $\mathbb{F}_p((t))$
and $p$-adic fields $\mathbb{Q}_p$.
Theorem~\ref{thm:intro2} is in a sense optimal, see the discussion at the end of this note.

\section{Defining subsets of the valuation ring}

\noindent
Let $K$ be a henselian valued field with valuation ring $\mathcal{O}\subseteq K$, maximal ideal $\mathfrak{m}\subseteq\mathcal{O}$ and residue field $F=\mathcal{O}/\mathfrak{m}$.
For $a\in\mathcal{O}$ we let $\bar{a}=a+\mathfrak{m}\in F$ be its residue class
and write $\bar{f}\in F[X]$ for the reduction of a polynomial $f\in\mathcal{O}[X]$.

We start by simplifying the key lemma of \cite{AnscombeKoenigsmann}, thereby generalizing it to arbitrary henselian valuations.
This proof follows Helbig \cite{Helbig}.
Here, and in what follows, by $f(K)^{-1}$ we mean the set $\{f(x)^{-1}:x\in K\}$ and implicitly claim that $f(x)\neq0$ for all $x\in K$.

\begin{Lemma}\label{PZ}
Let $f\in\mathcal{O}[X]$ be a monic polynomial such that $\bar{f}$ has no zero in $F$, and let $a\in K$. 
Let $U_{f,a}:=f(K)^{-1}-f(a)^{-1}$. Then the following holds:
\begin{enumerate}
\item[a)] $f(K)^{-1}\subseteq\mathcal{O}$
\item[b)] $U_{f,a}\subseteq\mathcal{O}$
\item[c)] If in addition $a\in\mathcal{O}$ and $f'(a)\notin\mathfrak{m}$, then $\mathfrak{m}\subseteq U_{f,a}$.
\end{enumerate}
\end{Lemma}

\begin{proof}
a)
We have that $f(K)\cap\mathfrak{m}=\emptyset$:
If $x\in K$ with $f(x)\in\mathfrak{m}$, then $x\in\mathcal{O}$ since $\mathcal{O}$ is integrally closed and $f$ is monic,
and hence $\bar{f}(\bar{x})=0$, contradicting the assumption that $\bar{f}$ has no zero in $F$.
Therefore, $f(K)^{-1}\subseteq(K\smallsetminus\mathfrak{m})^{-1}=\mathcal{O}$.

b) From a) we get that $f(K)^{-1}\subseteq\mathcal{O}$, and in particular $f(a)^{-1}\in\mathcal{O}$.
Thus, $U_{f,a}\subseteq\mathcal{O}$.

c)
Now assume that $a\in\mathcal{O}$ and $f'(a)\notin\mathfrak{m}$.
Let $x\in\mathfrak{m}$. 
Since $a\in\mathcal{O}$ we have $f(a)\in\mathcal{O}$, hence $f(a)\in\mathcal{O}^\times$.
Define $g(X)=f(X)-(f(a)+x)\in\mathcal{O}[X]$. 
Then $g(a)=-x\in\mathfrak{m}$ and $g'(a)=f'(a)\notin\mathfrak{m}$,
so by the assumption that $\mathcal{O}$ is henselian there exists $b\in\mathcal{O}$ with $g(b)=0$, i.e.~$f(a)+x=f(b)$. 
Hence, $f(a)+\mathfrak{m}\subseteq f(K)$.
Since $f(a)\in\mathcal{O}^\times$ we get that
$f(a)^{-1}+\mathfrak{m}=(f(a)+\mathfrak{m})^{-1}\subseteq f(K)^{-1}$,
and therefore $\mathfrak{m}\subseteq U_{f,a}$.
\end{proof}

We observe that one can get rid of the element $a$ even if it is not in the (model theoretic) algebraic closure of the prime field:

\begin{Lemma}\label{lemmag}
Let $f\in \mathcal{O}[X]$ be a monic polynomial such that $\bar{f}$ has no zero in $F$, and $a\in \mathcal{O}$ such that $f'(a)\notin\mathfrak{m}$.
Then $U:=f(K)^{-1}-f(K)^{-1}$ satisfies $\mathfrak{m}\subseteq U\subseteq\mathcal{O}$.
\end{Lemma}

\begin{proof}
By Lemma \ref{PZ}a, $f(K)^{-1}\subseteq\mathcal{O}$, hence $U\subseteq\mathcal{O}$. 
Since $a\in\mathcal{O}$ and $f'(a)\notin\mathfrak{m}$, Lemma~\ref{PZ}c implies that $\mathfrak{m}\subseteq U_{f,a}\subseteq U$.
\end{proof}

Clearly, $U$ can be defined in $K$ by the $\exists$-formula
$$
 \varphi_f(x) \equiv (\exists y,z,y_1,z_1)( x=y_1-z_1 \wedge y_1f(y)=1 \wedge z_1f(z)=1 ).
$$
Note that if $f\in\mathbb{Z}[X]$, then $\varphi_f$ is an $\exists$-$\emptyset$-formula.

\begin{Lemma}\label{idea}
If 
$U,T\subseteq\mathcal{O}$ are such that 
$\mathfrak{m}\subseteq U$
and $T$ meets all residue classes (i.e.~$\bar{T}=F$), then $\mathcal{O}=U+T$. 
\end{Lemma}

\begin{proof}
If for $x\in\mathcal{O}$ we let $t\in T$ with $\bar{t}=\bar{x}$, then
$x=u+t$ with $u:=x-t\in\mathfrak{m}\subseteq U$.
\end{proof}

Thus, if $\varphi$ defines $U$ and $\psi$ defines $T$, then
$$
 \eta(x)\equiv(\exists u,t)(x=u+t\wedge\varphi(u)\wedge\psi(t))
$$
defines $\mathcal{O}$. 
Note that if $\varphi$ and $\psi$ are $\exists$-$\emptyset$-formulas, then so is $\eta$.

We now give a first generalization of \cite[Theorem 1.1]{AnscombeKoenigsmann}.
We denote by $F_0$ the prime field of $F$ and by 
$F_{\rm alg}$ the algebraic closure of $F_0$ in $F$.
By abuse of notation we will consider polynomials $f\in\mathbb{Z}[X]$ as elements of $\mathcal{O}[X]$
via the canonical homomorphism $\mathbb{Z}\rightarrow\mathcal{O}$.

\begin{Lemma}\label{lem:trace}
For every prime $p$ and positive integer $m$ there exists
$f\in\mathbb{F}_p[X]$ monic, separable and irreducible of degree $m$ 
with $f'(0)\neq 0$.
\end{Lemma}

\begin{proof}
Let $q=p^m$.
Since $\mathbb{F}_q/\mathbb{F}_p$ is Galois it has a normal basis, i.e.~there exists $\alpha\in\mathbb{F}_q$
such that the conjugates of $\alpha$ form an $\mathbb{F}_p$-basis of $\mathbb{F}_q$.
In particular, $\alpha$ has degree $m$ and non-zero trace over $\mathbb{F}_p$.
Let $f\in\mathbb{F}_p[X]$ be the minimal polynomial of $\alpha^{-1}$.
Then $f$ is irreducible of degree $m$ and $f'(0)=\pm{\rm Tr}_{\mathbb{F}_{q}/\mathbb{F}_p}(\alpha)/{\rm N}_{\mathbb{F}_{q}/\mathbb{F}_p}(\alpha)\neq 0$.
\end{proof}

\begin{Lemma}\label{lem:finite}
If $F$ is finite,
then there exist 
$f\in F_0[X]$ monic, separable and irreducible which has no zero in $F$,
and $a\in F$ with $f'(a)\neq 0$.
\end{Lemma}

\begin{proof}
Identify $F_0=\mathbb{F}_p$,
let $m$ be any positive integer that does not divide $[F:F_0]$, choose $f$ of degree $m$ as in Lemma~\ref{lem:trace},
and let $a=0$.
\end{proof}

\begin{Theorem}\label{thm:finite}
Let $K$ be a henselian valued field with valuation ring $\mathcal{O}$ and residue field $F$.
If $F$ is finite, then there exists an $\exists$-$\emptyset$-definition of $\mathcal{O}$ in $K$.
\end{Theorem}

\begin{proof}
If $F=\mathbb{F}_q$, let $g=X^q-X\in\mathbb{Z}[X]$ and $\psi(x)\equiv (g(x)=0)$.
Since $\bar{g}'=-1$, the assumption that $\mathcal{O}$ is henselian gives that $T:=\psi(K)\subseteq\mathcal{O}$
is a set of representatives of $F$. In particular, it meets all residue classes.
Choose $f\in F_0[X]$ as in Lemma \ref{lem:finite}
and let $\tilde{f}\in\mathbb{Z}[X]$ be a monic lift of $f$.
Since there exists $a\in F$ with $f'(a)\neq 0$, 
a lift $\tilde{a}\in\mathcal{O}$ of $a$ satisfies $\tilde{f}'(\tilde{a})\notin\mathfrak{m}$.
Let $\varphi\equiv\varphi_f$.
By Lemma \ref{lemmag},
$U:=\varphi(K)$ satisfies $\mathfrak{m}\subseteq U\subseteq\mathcal{O}$.
Therefore, Lemma \ref{idea}
shows that $\eta(K)=\mathcal{O}$. 
\end{proof}

\section{Pseudo-algebraically closed residue fields}

\noindent
We now consider assumptions on the residue field $F$ under which we can
define a set $T$ as in Lemma \ref{idea}.
For basics on pseudo-algebraically closed (PAC) fields we refer to \cite[Chapter~11]{FJ}.
For $d\in\mathbb{N}$ we fix the constant $c(d)=(2d-1)^4$.

\begin{Lemma}\label{PAC1}
Let $f\in F[X]$ be non-constant and square-free (over the algebraic closure). 
Then $F=f(F)f(F)\cup\{0\}$ if 
$F$ is PAC or $F$ is finite with $|F|>c({\rm deg}(f))$.
\end{Lemma}

\begin{proof}
Let $0\neq c\in F$.
One checks that the polynomial $f(X)f(Y)-c\in F[X,Y]$ is absolutely irreducible, cf.~\cite[Proposition 1.1]{HP}. 
Thus, if $F$ is PAC we can conclude that there exist $x,y\in F$ with $f(x)f(y)-c=0$, i.e.~$c\in f(F)f(F)$.
If $F$ is finite with $|F|>c({\rm deg}(f))$ we come to the same conclusion
by applying the Hasse-Weil bound, cf.~\cite[Corollary 5.4.2]{FJ}.
\end{proof}

\begin{Lemma}\label{PAC2}
Let $f\in \mathcal{O}[X]$ be monic such that $\bar{f}$ is square-free and has no zero in $F$.
Then $T:=f(K)^{-1}f(K)^{-1}\cup\{0\}\subseteq\mathcal{O}$.
If in addition $F$ is PAC or finite with $|F|>c({\rm deg}(f))$,
then $T$ meets all residue classes.
\end{Lemma}

\begin{proof}
By Lemma \ref{PZ}a, $f(K)^{-1}\subseteq\mathcal{O}$, hence $T\subseteq\mathcal{O}$.
If $F$ is PAC or finite with $|F|>c({\rm deg}(f))$, then,
since $F^\times\subseteq \bar{f}(F)\bar{f}(F)$ by Lemma \ref{PAC1}, 
also 
$$
 F^\times\subseteq (\bar{f}(F)\bar{f}(F))^{-1}\subseteq(\overline{f(\mathcal{O})}\cdot\overline{f(\mathcal{O})})^{-1}\subseteq \overline{f(K)^{-1}f(K)^{-1}},
$$ 
hence $T$ satisfies $F=\overline{T}$.
\end{proof}

Clearly, the set $T$ can be defined in $K$ by the $\exists$-formula
$$
 \psi_f(x) \equiv (\exists y,z,y_1,z_1)(x=0\vee(x=y_1z_1\wedge y_1f(y)=1\wedge z_1f(z)=1)).
$$
Let
$$
 \eta_f(x)\equiv(\exists u,t)(x=u+t\wedge\varphi_f(u)\wedge\psi_f(t)).
$$

\begin{Proposition}\label{prop}
Let $f\in \mathcal{O}[X]$ be monic such that $\bar{f}$ is square-free and has no zero in $F$.
Then $\eta_f(K)\subseteq\mathcal{O}$.
If in addition
there exists $a\in \mathcal{O}$ such that $f'(a)\notin\mathfrak{m}$
and $F$ is PAC or finite with $|F|>c({\rm deg}(f))$, then
$\eta_f(K)=\mathcal{O}$.
\end{Proposition}

\begin{proof}
Let $U=\varphi_f(K)$, so $U\subseteq\mathcal{O}$ by Lemma \ref{PZ}.
By Lemma \ref{PAC2}, $T:=\psi_f(K)\subseteq\mathcal{O}$,
so $\eta_f(K)=U+T\subseteq\mathcal{O}$.
If in addition
there exists $a\in \mathcal{O}$ such that $f'(a)\notin\mathfrak{m}$
and
 $F$ is PAC or finite with $|F|>c({\rm deg}(f))$, then
Lemma \ref{lemmag} gives that $\mathfrak{m}\subseteq U$, and
Lemma \ref{PAC2} gives that $T$ meets all residue classes,
hence $\eta_f(K)=\mathcal{O}$ by Lemma \ref{idea}.
\end{proof}

\begin{Lemma}\label{finite}
If $F$ is infinite and $F_{\rm alg}$ is not algebraically closed, then there exist 
$f\in F_0[X]$ monic, separable, and irreducible which has no zero in $F$,
and $a\in F$ with $f'(a)\neq 0$.
\end{Lemma}

\begin{proof}
Since $F_{\rm alg}$ is not algebraically closed, there exists
a monic irreducible $f\in F_0[X]$ which has no zero in $F_{\rm alg}$, hence in $F$.
Since $F_0$ is perfect, $f$ is separable, hence $f'\neq 0$.
Therefore, since $F$ is infinite, there exists $a\in F$ with $f'(a)\neq 0$.
\end{proof}

\begin{Theorem}\label{thm:PAC}
Let $K$ be a henselian valued field with valuation ring $\mathcal{O}$ and residue field $F$.
If $F$ is pseudo-algebraically closed and $F_{\rm alg}$ is not algebraically closed,
then there exists an $\exists$-$\emptyset$-definition of $\mathcal{O}$ in $K$.
\end{Theorem}

\begin{proof}
Choose $f\in F_0[X]$ as in Lemma \ref{finite}
and let $\tilde{f}\in\mathbb{Z}[X]$ be a monic lift of $f$.
Since there exists $a\in F$ with $f'(a)\neq 0$, 
a lift $\tilde{a}\in\mathcal{O}$ of $a$ satisfies $\tilde{f}'(\tilde{a})\notin\mathfrak{m}$.
By Proposition~\ref{prop},
$\eta_{\tilde{f}}(K)=\mathcal{O}$.
\end{proof}

\begin{Corollary}
Let $K$ be a henselian valued field with valuation ring $\mathcal{O}$ and residue field $F$.
If $F$ is pseudo-real closed and $F_{\rm alg}$ is neither real closed nor algebraically closed,
then there exists an $\exists$-$\emptyset$-definition of $\mathcal{O}$ in $K$.
\end{Corollary}

\begin{proof}
Let $K'=K(\sqrt{-1})$. Then the residue field
$F'=F(\sqrt{-1})$ of $K'$ is PAC by \cite{Prestel},
and $F'_{\rm alg}=F_{\rm alg}(\sqrt{-1})$ is not algebraically closed by the Artin-Schreier theorem.
By Theorem \ref{thm:PAC} there exists an $\exists$-$\emptyset$-definition of the unique prolongation $\mathcal{O}'$ of $\mathcal{O}$ in $K'$.
By interpreting $K'$ in $K$ we get an $\exists$-$\emptyset$-definition of $\mathcal{O}=\mathcal{O}'\cap K$ in $K$.
 \end{proof}

\begin{Remark}
Note that as soon as $F$ is infinite we cannot hope to have an $\exists$-$\emptyset$-definition of
a {\em set of representatives} $T\subseteq\mathcal{O}$ of $F$: For example,
if $K=F((t))$, then $F$ is never $\exists$-$\emptyset$-definable in $K$ unless it is finite, cf.~\cite[Corollary 9]{Fehm}.
This explains why we rather define a set $T\subseteq\mathcal{O}$ that {\em meets all residue classes}.
\end{Remark}

\begin{Remark}
We point out that the assumption that $F_{\rm alg}$ {\em is not algebraically closed} in Theorem~\ref{thm:PAC} is indeed necessary.
For example, 
let $K$ be the field of generalized power series $F((\mathbb{Q}))$ over a field $F$.
If $F_{\rm alg}$ is algebraically closed, then so is $K':=F_{\rm alg}((\mathbb{Q}))$, cf.~\cite[18.4.3]{Efrat}.
Therefore, $K'$ is existentially closed in $K$.
So, if $\varphi$ is an $\exists$-$\emptyset$-definition of the valuation ring in $K$, then
$\varphi(K') = \varphi(K)\cap K'$
is a non-trivial valuation ring, contradicting the fact that definable subsets
of an algebraically closed field are finite or cofinite.
\end{Remark}

\section{Uniform definitions}
\label{sec:uniform}

\noindent
We now deal with definitions which are uniform over certain families of finite residue fields.
We start with an example in fixed residue characteristic $p$:

\begin{Theorem}\label{thm:uniformk}
For all prime numbers $p$ and positive integer $m$ there exists an $\exists$-$\emptyset$-formula $\varphi$
such that $\varphi(K)=\mathcal{O}$
for all henselian valued fields $K$ with valuation ring $\mathcal{O}$
and residue field $F=\mathbb{F}_{p^n}$ with $m\notdivides n$.
\end{Theorem}

\begin{proof}
Assume that $F=\mathbb{F}_{p^n}$ with $m\notdivides n$.
Choose $f\in\mathbb{F}_p[X]$ irreducible of degree $m$ as in Lemma \ref{lem:trace}.
Then $f$ has no zero in $F$
and
there exists $a\in F$ with $f'(a)\neq0$.
Let $\tilde{f}\in\mathbb{Z}[X]$ be a monic lift of $f$.
By Proposition \ref{prop}, 
$\eta_{\tilde{f}}(K)\subseteq\mathcal{O}$,
and
$\eta_{\tilde{f}}(K)=\mathcal{O}$
for
$p^{n}>c(m)$.
For $k\in\mathbb{N}$ with $m\notdivides k$ let $\psi_k(x)\equiv (x^{p^k}-x=0)$
and let 
$$
 \eta_k(x)\equiv(\exists u,t)(x=u+t\wedge\varphi_{\tilde{f}}(u)\wedge\psi_k(t)).
$$
As in the proof of Theorem \ref{thm:finite} we see that $\eta_k(K)\subseteq\mathcal{O}$, and $\eta_k(K)=\mathcal{O}$ if $n=k$.
Therefore, with $M=\{k\in\mathbb{N} : m\notdivides k\mbox{ and }p^k\leq c(m)\}$,
$$
 \varphi(x) \equiv \eta_{\tilde{f}}(x)\vee\bigvee_{k\in M}\eta_k(x)
$$
satisfies $\varphi(K)=\mathcal{O}$ for all $n$ with $m\notdivides n$.
\end{proof}

\begin{Remark}
The condition $m\notdivides n$ in Theorem~\ref{thm:uniformk} is indeed necessary:
If a $\exists$-$\emptyset$-formula $\varphi$ defines $\mathbb{F}_{p^n}[[t]]$ in $\mathbb{F}_{p^n}((t))$
for all $n$ in a set $M$, then there is some $m\in\mathbb{N}$ such that $m\notdivides n$ for all $n\in M$:
Otherwise, $\bigcup_{n\in M}\mathbb{F}_{p^n}$ would equal the algebraic closure of $\mathbb{F}_p$,
so since every finite extension of $\mathbb{F}_{p^n}((t))$ is isomorphic to $\mathbb{F}_{p^{n'}}((t))$ for some $n|n'$,
we would get a definition of a non-trivial valuation ring in the algebraic closure
of $\mathbb{F}_{p}((t))$, cf.~\cite[Theorem 4]{Cluckersetal}, which is impossible.
\end{Remark}

We now turn to uniformity in $p$.
Let $\mathbb{P}$ denote the set of all odd prime numbers.
For a subset $P\subseteq\mathbb{P}$, we denote by $\delta(P)$ the Dirichlet-density of $P$, if it exists.
For a formula $\varphi$ let
$$
 \mathbb{P}(\varphi) = \{ p\in\mathbb{P} : \varphi(\mathbb{Q}_p)=\mathbb{Z}_p \} 
$$
and let
$\mathbb{P}'(\varphi)$ be the set of $p\in\mathbb{P}$ such that
$\varphi(K)=\mathcal{O}$ for all henselian valued fields $K$ with valuation ring $\mathcal{O}$ and residue field $F=\mathbb{F}_p$. 
We have that $\mathbb{P}'(\varphi)\subseteq\mathbb{P}(\varphi)$,
and it is known that $\mathbb{P}(\varphi)$ has a Dirichlet-density for every formula $\varphi$, 
cf.~\cite[Theorem 16]{Axfinite}, \cite[Theorem 20.9.3]{FJ}.
It is also known that $\mathbb{P}(\varphi)$ differs from 
$\{ p\in\mathbb{P} : \varphi(\mathbb{F}_p((t)))=\mathbb{F}_p[[t]] \}$ only by a finite set, see \cite[p.~606]{AxKochen},
so for all results concerning Dirichlet density we could as well use $\mathbb{F}_p((t))$ instead of $\mathbb{Q}_p$.

\begin{Theorem}\label{thm:uniform}
For every $\epsilon>0$ there exists an $\exists$-$\emptyset$-formula $\varphi$ such that
$\delta(\mathbb{P}'(\varphi))>1-\epsilon$.
\end{Theorem}

\begin{proof}
For $n\in\mathbb{N}$ let $f_n=X^2-n\in\mathbb{Z}[X]$
and 
$$
 P_n=\left\{p\in\mathbb{P}:\left(\frac{n}{p}\right)=-1\right\}=\left\{p\in\mathbb{P} : \mathbb{F}_p\not\models(\exists y)(y^2=n) \right\}.
$$ 
Note that 
if $K$ is henselian valued with residue field $F=\mathbb{F}_p$, then $p\in P_n$ if and only if $K\not\models(\exists y)(y^2=n)$.
If $p\in P_n$ with $p>c(2)$, then $p\in\mathbb{P}'(\eta_{f_n})$ by Proposition \ref{prop}.
By the quadratic reciprocity law and Dirichlet's theorem,
there exists $N\in\mathbb{N}$ such that for $P=\bigcup_{n=2}^N P_n$ we have $\delta(P)>1-\epsilon$.
Let 
$$
 \varphi_n(x)\equiv (\exists y)(y^2=n)\vee\eta_{f_n}(x)
$$
and $\varphi(x)\equiv \bigwedge_{n=2}^N \varphi_n(x)$.
Let $p\in \mathbb{P}$ which lies in the open interval $I:=(c(2),\infty)$.
If $p\in P_n$, then $\varphi_n(K)=\mathcal{O}$, otherwise $\varphi_n(K)=K$.
Thus, $\varphi(K)=\bigcap_{n=2}^N \varphi_n(K)=\mathcal{O}$ if $p\in P$, and $\varphi(K)=K$ otherwise.
So, if $p\in P$, then
$p\in\mathbb{P}'(\varphi)\subseteq\mathbb{P}(\varphi)$,
and if $p\notin P$, then $p\notin\mathbb{P}(\varphi)$.
Thus, $\mathbb{P}'(\varphi)\cap I=\mathbb{P}(\varphi)\cap I=P\cap I$,
and therefore
$\delta(\mathbb{P}'(\varphi))=\delta(\mathbb{P}(\varphi))=\delta(P)>1-\epsilon$.
\end{proof}

On the other hand, 
the proof of \cite[Theorem 5]{Cluckersetal} shows the following:

\begin{Proposition}\label{thm:Cluckers}
Let $P$ be a set of prime numbers with $\delta(P)=1$.
Then there exists no $\exists$-$\emptyset$-formula $\varphi$ such that
$P\subseteq\mathbb{P}(\varphi)$.
\end{Proposition}

This also explains that Theorem \ref{thm:uniform} cannot be strengthened to give a uniform $\exists$-$\emptyset$-definition
for {\em every} set $P$ with $\delta(P)<1$:

\begin{Proposition}
There exists a set $P$ of prime numbers with $\delta(P)=0$ for which there exists no $\exists$-$\emptyset$-formula $\varphi$
such that $P\subseteq\mathbb{P}(\varphi)$.
\end{Proposition}

\begin{proof}
List all  $\exists$-$\emptyset$-formulas as $\varphi_1,\varphi_2,\dots$ and let $N=\{\ell_1,\ell_2,\dots\}\subseteq\mathbb{P}$
be any infinite set with $\delta(N)=0$.
Proposition \ref{thm:Cluckers} (or rather \cite[Theorem 5]{Cluckersetal} directly) implies that
for each $i$, $\mathbb{P}(\varphi_i)$ is not cofinite in $\mathbb{P}$.
Therefore, we can choose some $p_i\in\mathbb{P}$ with $p_i>\ell_i$ and $p_i\notin\mathbb{P}(\varphi_i)$. 
Then $P=\{p_1,p_2,\dots\}$ has $\delta(P)\leq\delta(N)=0$, but $P\not\subseteq\mathbb{P}(\varphi_i)$ for each $i$.
\end{proof}

\section*{Acknowledgements}

\noindent
The author would like to thank Will Anscombe, Patrick Helbig, Franziska Jahnke and Jochen Koenigsmann
for interesting and productive discussions on this subject,
and Moshe Jarden for helpful comments on a previous version of the manuscript.

\end{document}